\documentclass[11pt]{article}
\usepackage{amssymb,amsfonts,amsmath,latexsym,epsf,tikz,url}
\usepackage[usenames,dvipsnames]{pstricks}
\usepackage{pstricks-add}
\usepackage{epsfig}
\usepackage{pst-grad} 
\usepackage{pst-plot} 
\usepackage[space]{grffile} 
\usepackage{etoolbox} 
\usepackage{float}
\usepackage{soul}
\usepackage{tikz}
\usepackage[colorinlistoftodos]{todonotes}
\usepackage{pgfplots}
\usepackage{mathrsfs}
\usepackage[colorlinks]{hyperref}
\usetikzlibrary{arrows}
\makeatletter 
\patchcmd\Gread@eps{\@inputcheck#1 }{\@inputcheck"#1"\relax}{}{}
\makeatother

\newtheorem{theorem}{Theorem}[section]

\newtheorem{proposition}[theorem]{Proposition}
\newtheorem{observation}[theorem]{Observation}

\newtheorem{lemma}[theorem]{Lemma}
\newtheorem{remark}[theorem]{Remark}

\newtheorem{definition}[theorem]{Definition}

\begin{document}

\def\nt{\noindent}

\title{Some resolving parameters with the minimum size for two specific graphs }

\author{	
\small Ali Zafari$^{1}$
\and  
\small Saeid Alikhani$^{2}$
}



\maketitle

\begin{center}

$^1$Department of Mathematics, Faculty of Science,
Payame Noor University, P.O. Box 19395-4697, Tehran, Iran\\

\medskip
$^{2}$Department of Mathematical Sciences, Yazd University, 89195-741, Yazd, Iran\\

\medskip
{\tt zafari.math@pnu.ac.ir~~ alikhani@yazd.ac.ir}
\end{center}

\begin{abstract}
A resolving set for a graph $G$ is a set of vertices $Q = \{q_1, ..., q_k\}$ such that, for all $p\in V(G)$ 
the $k$-tuple $(d(p, q_1), ..., d(p, q_k ))$ uniquely determines $p$,
where $d(p, q_i)$ is considered as the minimum length of a shortest path from $p$ to $q_i$ in graph $G$.
In this paper, we consider the computational study of some resolving sets with the minimum size for the $m$-cylinder graph $(C_n\Box P_k)\Box P_m$. 
The Boolean lattice $BL_n$, $n\geq 1$, is the graph whose vertex set is the set of all subsets of $[n]=\{1,2,...,n\}$,
where two subsets $X$ and $Y$ are adjacent if their symmetric difference has precisely
one element. In the graph $BL_n$, the layer $L_i$ is the family of $i$-subsets of $[n]$. The subgraph $BL_n(i,i+1)$ is the subgraph of $BL_n$ induced by layers $L_i$ and $L_{i+1}$.
Usually the graph $BL_n(1,2)$ is denoted by $H(n)$. We study the minimum size of a resolving set, doubly resolving set and strong resolving set for the graph $L(n)$, which is the line graph of $H(n)$.
\end{abstract}

\noindent{\bf Keywords:}   resolving set; doubly resolving set; strong resolving set; line graph.
  
\medskip
\noindent{\bf AMS Subj.\ Class.:} 05C12, 05C76.


\section{Introduction}
\label{sec:introduction}
All graphs considered in this paper are assumed to be finite and connected. For notation and terminology not defined here,
see 
~\cite{pap-cg-1}.
A graphical representation of a vertex $p$ of a connected graph $G$ relative to an arranged subset $Q = \{q_1, ..., q_k\}$ of vertices of $G$ is defined as the $k$-tuple $(d(p, q_1), ..., d(p, q_k ))$, and this $k$-tuple is denoted by $r(p|Q)$, where $d(p, q_i)$ is considered as the minimum length of a shortest path from $p$ to $q_i$ in the graph $G$. If any vertices $p$ and $q$ that belong to  $V(G)\setminus Q$  have various representations with respect to the set $Q$, then $Q$ is called a resolving set for $G$ 
\cite{pap-psb.gh1-2}.
Slater 
\cite{pap-pjs-1}
considered the concept and notation of the metric dimension problem under the term locating set. Also, Harary and Melter 
\cite{pap-fh-1}
considered these problems under the term metric dimension as follows.
A resolving set of vertices of graph $G$ with the minimum size or cardinality is called the metric dimension of $G$ and this minimum size denoted by $\beta(G)$. However, some results for determining the metric dimension of various families of graphs have been obtained,
such as Cayley digraphs of split metacyclic groups 
\cite{pap-M.Abas.ET.AL},
Mobius ladders  
\cite{pap-M.Ali.ET.AL},
regular bipartite graphs 
\cite{pap-M.Baca.ET.AL},
lexicographic product of graphs 
\cite{pap-mj-1},
certain families of Toeplitz graphs 
\cite{pap-J.B.Liu.ET.AL},
layer sun graph and line graph of the layer sun graph
\cite{pap-jb1-1},
and Crystal Cubic Carbon $CCC(n)$ 
\cite{pap-X.Zhang.ET.AL}.

In 2007 C\'{a}ceres et al. 
\cite{pap-jc-1}
considered the concept and notation of a doubly resolving set of graph $G$ as follows. For any two
vertices $u$ and $v$ in a graph $G$, we say that $u$ and $v$ are doubly resolved by $x, y \in V (G)$, if $d(u,x)-d(u,y)\neq d(v,x)-d(v,y)$,
and we can see that a subset $Q = \{q_1, ..., q_l\}$ of vertices of a graph $G$ is a doubly resolving set of $G$, if for any various vertices $x$ and $y$ in $G$ we have $r(x|Q)-r(y|Q)\neq\lambda I$, where $\lambda$ is an  integer, and $I$ denotes   the unit $l$- vector $(1,..., 1)$, see 
\cite{pap-aa.s.s1-2}.
The size of the smallest doubly resolving set of vertices in a graph $G$, is denoted by $\psi(G)$. For more results about doubly resolving set, see \cite{pap-chr-1,pap-mj-1.1,pap-az-1}.

The notion of a strong metric dimension problem set of vertices of graph $G$ was introduced by  Seb\"{o} and Tannier 
\cite{pap-as-1},
indeed introduced a more restricted invariant than the metric dimension and this was further investigated by Oellermann and Peters-Fransen 
\cite{pap-oro-1}.
A vertex $u$ of a graph $G$ is called maximally distant from a vertex $v$ of $G$, if for every $w\in N_G(u)$, we have  $d(v,w) \leq d(v, u)$, where $N_G(u)$ denote the set of neighbors that $u$ has in $G$. If $u$ is maximally distant from $v$ and $v$ is maximally distant from $u$, then $u$ and $v$ are said to be mutually maximally distant.
A set $Q\subseteq V(G)$  is called  strong resolving set of $G$, if for any various vertices $p$ and $q$ of $G$ there is a vertex of $Q$, say $r$  so that $p$ belongs to a shortest $q - r$ path or $q$ belongs to a shortest $p - r$ path.
A strong metric basis of $G$, denoted  by $sdim(G)$, is the size of the smallest strong resolving set of vertices in a graph $G$. For more results about this parameter, see 
~\cite{pap-J.Kratica.ET.AL,pap-Kuziak,pap-ja-1}.

The written papers about the resolving parameters for graphs, show that these parameters have a very important potential to solve a number of representative real life problems, which have been described in several works. For instance, they have been frequently used in graph theory, chemistry, coding theory, robotics and many other disciplines.  Note that these problems are NP-hard, see 
~\cite{pap-pj.jh-1,pap-X. Chen1-2,pap-Khuller,pap-J.Kratica.ET.AL1}.

 The Cartesian  product of two graphs $G$ and $H$, is denoted by $G \Box H$, is a graph with vertex set $V(G) \times V(H)$ and with edge set
$E(G \times H)$ so that $(g_1,h_1)(g_2,h_2) \in E(G \Box H)$, whenever $h_1 = h_2$ and $g_1g_2 \in E(G)$, or $g_1 = g_2$ and $h_1h_2 \in E(H)$. 
We denote by $C_n$ and $P_k$  the cycle on $n \geq 3$ and the path on $k\geq 3$ vertices, respectively.
Also, we use $C_n\Box P_k$ and $(C_n\Box P_k)\Box P_m$ to denote the cylinder graph and the $m$-cylinder graph, respectively.

The Boolean lattice $BL_n$, $n\geq 1$, is the graph whose vertex set is the set of all subsets of $[n]=\{1,2,...,n\}$,
where two subsets $X$ and $Y$ are adjacent if their symmetric difference has precisely
one element. In the graph $BL_n$, the layer $L_i$ is the family of $i$-subsets of $[n]$. The subgraph $BL_n(i,i+1)$ is the subgraph of $BL_n$ induced by layers $L_i$ and $L_{i+1}$.
Usually the graph $BL_n(1,2)$ is denoted by $H(n)$. We study the minimum size of a resolving set, doubly resolving set and strong resolving set for the graph $L(n)$, which is the line graph of $H(n)$ (see \cite{pap-sm-1}).

\medskip
In this paper, we consider the computational study of some resolving sets with the minimum size for the $m$-cylinder graph $(C_n\Box P_k)\Box P_m$ and the graph  $L(n)$. 
\section{Results for the $m$-cylinder graph $(C_n\Box P_k)\Box P_m$ }

 In this section, we consider a class of the form  $(C_n\Box P_k)\Box P_m$, which we call it  $m$-cylinder graph and we compute
some resolving sets for the $m$-cylinder graph $(C_n\Box P_k)\Box P_m$. In fact,
some resolving parameters of the cylinder graph $C_n\Box P_k$ have been calculated.
Detailed information and further references concerning the resolving parameters of the cylinder graph $C_n\Box P_k$ can  be found in 
\cite{pap-jc-1,pap-skbp,pap-kn-1,pap-ja-1}. We compute some resolving parameters of the cylinder graph $C_n\Box P_k$ with another approach. 
Also, we study some resolving sets for the $m$-cylinder graph $(C_n\Box P_k)\Box P_m$.
For more results of families of graphs with constant metric, see
\cite{pap-M.Ahmad.ET.AL,pap-mi-1}.
 
\begin{observation}{\rm\cite{pap-jc-1}}
For  any connected graph $G$ and the path $P_n$ with $n$ vertices,
$\beta(G \Box P_n)\leq \beta(G)+1$.
\end{observation}

In the following, we construct a class of graphs so that this class is isomorphic to $(C_n\Box P_k)$, and call this graph cylinder graph.

\begin{definition}
Let $n$ and $k$ be  fixed positive integers, with $n,k\geq3$  and $[n]=\{1, ..., n\}$.
Suppose that $G$ is a graph with vertex set $\{x_1,  ..., x_{nk}\}$ on the layers $V_1, V_2, ..., V_k$, where $V_p=\{x_{(p-1)n+1}, x_{(p-1)n+2}, ..., x_{(p-1)n+n}\}$ for $p\in\{1,...,k\}$,
and the edge set of graph $G$ is
\begin{align*}
E(G)&=\{x_ix_j\,|\, x_i, x_j\in V_p,  1\leq i<j\leq nk, j-i=1 \text{or}  ~   j-i=n-1  \}\\
&\cup\{x_ix_j\,|\, x_i\in V_q, x_j\in V_{q+1},  1\leq i<j\leq nk,  1\leq q\leq k-1, j-i=n  \}.
\end{align*}
We can see that this graph is isomorphic to the cylinder graph $C_n\Box P_k$. So, we can assume throughout this article
$V(C_n\Box P_k)=\{x_1,  ..., x_{nk}\}$, and we use $V_p$ to denote a layer of cylinder graph $C_n\Box P_k$, where $V_p$, is defined already.
\end{definition}

Now, we give a more elaborate description of cylinder graph $C_n\Box P_k$,
that are required to prove of Theorems.
\begin{remark}
 Let $e$ and $d$ be positive integers with $1\leq e<d\leq nk$. We say that two vertices $x_e$ and  $x_d$ in $C_n\Box P_k$ are compatible, if $n|d-e$.

\end{remark}

In the following, we construct a class of graphs so that this class is isomorphic to $(C_n\Box P_k)\Box P_m$, and call this graph $m$-cylinder graph.

\begin{definition}
Let $m\geq2$ be an integer. Suppose $1\leq i\leq m$ and consider $i^{th}$ copy of cylinder graph $C_n\Box P_k$ with the vertex set
$\{x_1^{(i)},  ..., x_{nk}^{(i)}\}$ on the layers $V_1^{(i)}, V_2^{(i)}, ..., V_k^{(i)}$, where it can be defined $V_p^{(i)}$ as similar $V_p$ on the vertex set $\{x_1^{(i)},  ..., x_{nk}^{(i)}\}$. Also, suppose that $K$ is a graph with vertex set
$\{x_1^{(1)},  ..., x_{nk}^{(1)}\}\cup ...\cup \{x_1^{(m)},  ..., x_{nk}^{(m)}\}$ so that the vertex $x_t^{(r)}$ is adjacent to the vertex $x_t^{(r+1)}$ in $K$, for $1\leq t \leq nk$, and $1\leq r\leq m-1$. We can see that the graph $K$ is isomorphic to $(C_n\Box P_k)\Box P_m$, and we call this graph $m$-cylinder graph.
\end{definition}

\begin{remark}
Let $e$ and $d$ be positive integers with $1\leq e<d\leq nk$. We say that two vertices $x_e^{(i)}$ and  $x_d^{(i)}$ in $i^{th}$ copy of $C_n\Box P_k$ are compatible, if $n|d-e$. 
\end{remark}
\begin{theorem}\label{dd.1}
If $n\geq3$ is an odd integer, then the minimum size of a doubly resolving set of vertices for the cylinder graph $C_n\Box P_k$ is $3$.
\end{theorem}
\begin{proof}
We first consider the cylinder graph $C_n\Box P_k$ with the vertex set  $\{x_1,  ..., x_{nk}\}$ on the layers $V_1, V_2, ..., V_k$. Based on \cite{pap-jc-1}, if $n$ is an odd integer, then the minimum size of a resolving set in $C_n\Box P_k$ is $2$. In particular,
we can show that if $n$ is an odd integer, then all the elements of every minimum resolving set of vertices in $C_n\Box P_k$ must lie in
exactly one of the layers $V_1$ or $V_k$. Without lack of theory if we consider the layer $V_1$ of $C_n\Box P_k$, then we can show that all the minimum resolving sets of vertices in the layer $V_1$ of $C_n\Box P_k$ are the sets in the  forms
$M_i=\{x_i, x_{\lceil\frac{n}{2}\rceil+i-1}\}$, $1\leq i\leq \lceil\frac{n}{2}\rceil$ and $N_j=\{x_j, x_{\lceil\frac{n}{2}\rceil+j}\}$,
$1\leq j\leq \lfloor\frac{n}{2}\rfloor$. On the other hand, the arranged subsets $M_i$ and  $N_j$ of vertices in $C_n\Box P_k$ cannot be doubly resolving sets for the cylinder graph $C_n\Box P_k$, because for $1\leq i\leq \lceil\frac{n}{2}\rceil$ and two compatible vertices $x_{i+n}$ and $x_{i+2n}$ with respect to $x_i$, we have $r(x_{i+n}|M_i)-r(x_{i+2n}|M_i)=-I$, where $I$ denotes  the unit $2$-vector $(1, 1)$. It is also clear that for $1\leq j\leq \lfloor\frac{n}{2}\rfloor$ the arranged subsets $N_j$ of vertices in $C_n\Box P_k$ cannot be doubly resolving sets for the cylinder graph $C_n\Box P_k$, and so the minimum size of a doubly resolving set in $C_n\Box P_k$ is greater than $2$. Now, we claim that the minimum size of a doubly resolving set of vertices in $C_n\Box P_k$ is $3$. So it is  enough to find a doubly resolving set of vertices in $C_n\Box P_k$ with the size $3$. For each $1\leq i\leq \lceil\frac{n}{2}\rceil$, if $A_i=M_i\cup x_c=\{x_i, x_{\lceil\frac{n}{2}\rceil+i-1}, x_c\}$ is an arranged subset of vertices in $C_n\Box P_k$, where $x_c$ lies in the layer $V_k$ of $C_n\Box P_k$ and $x_c$ is a compatible vertex with respect to $x_i$, then
$A_i$ is a doubly resolving set of vertices in $C_n\Box P_k$, because it is sufficient to show that for any compatible vertices $x_e$ and $x_d$ in $C_n\Box P_k$, $r(x_e|A_i)-r(x_d|A_i)\neq\lambda I$. Suppose $x_e\in V_p$ and $x_d\in V_q$ are compatible vertices in
$C_n\Box P_k$, $1\leq p<q\leq k$. So, $r(x_e|M_i)-r(x_d|M_i)=-\lambda I$, where $\lambda$ is a positive integer, and $I$ indicates the unit $2$-vector $(1, 1)$. We also see that for the compatible vertex $x_c$ with respect to $x_i$,  $r(x_e|x_c)-r(x_d|x_c)=\lambda$. Therefore, $r(x_e|A_i)-r(x_d|A_i)\neq\lambda I$, where $I$ indicates the unit $3$-vector $(1, 1, 1)$, and the claim is proved. \newline
\end{proof}
\begin{theorem}\label{dd.2}
If $n\geq3$ is an odd integer, then the minimum size of a doubly resolving set of vertices for the $m$-cylinder graph $(C_n\Box P_k)\Box P_m$ is $4$.
\end{theorem}
\begin{proof}
Suppose the $m$-cylinder graph $(C_n\Box P_k)\Box P_m$ is a graph with vertex set
$\{x_1^{(1)},  ..., x_{nk}^{(1)}\}\cup ...\cup \{x_1^{(m)},  ..., x_{nk}^{(m)}\}$
so that the vertex
$x_t^{(r)}$ is adjacent to $x_t^{(r+1)}$ in $(C_n\Box P_k)\Box P_m$, for  $1\leq t \leq nk$, and $1\leq r\leq m-1$. We also see that for any connected graph $G$ and the path $P_m$, $\beta(G \Box P_m)\leq \beta(G)+1$. So, by considering $G=(C_n\Box P_k)$ we have $\beta(G \Box P_m)\leq \beta(G)+1=3$. Moreover, it is not hard to see that
for $1\leq i\leq \lceil\frac{n}{2}\rceil$, and $1\leq j\leq \lfloor\frac{n}{2}\rfloor$, $1^{th}$ copies of the arranged sets
$A_i=\{x_i, x_{\lceil\frac{n}{2}\rceil+i-1}, x_c\}$ and $B_j=\{x_j, x_{\lceil\frac{n}{2}\rceil+j}, x_c\}$, denoted by the sets
$A_i^{(1)}=\{x_i^{(1)}, x_{\lceil\frac{n}{2}\rceil+i-1}^{(1)}, x_c^{(1)}\}$ and
$B_j^{(1)}=\{x_j^{(1)}, x_{\lceil\frac{n}{2}\rceil+j}^{(1)}, x_c^{(1)}\}$, respectively, are resolving sets with the minimum size for the $m$-cylinder graph $(C_n\Box P_k)\Box P_m$.
On the other hand for $1\leq i\leq \lceil\frac{n}{2}\rceil$ the arranged sets
$A_i^{(1)}=\{x_i^{(1)}, x_{\lceil\frac{n}{2}\rceil+i-1}^{(1)}, x_c^{(1)}\}$ of vertices of $(C_n\Box P_k)\Box P_m$, cannot be doubly resolving sets for the $m$-cylinder graph $(C_n\Box P_k)\Box P_m$, because for each vertex $x_t^{(m)}$ of $(C_n\Box P_k)\Box P_m$, we have
$$r(x_t^{(m)}|A_i^{(1)})=( d(x_t^{(1)}, x_i^{(1)})+(m-1), d(x_t^{(1)}, x_{\lceil\frac{n}{2}\rceil+i-1}^{(1)})+(m-1), d(x_t^{(1)}, x_{c}^{(1)})+(m-1)).$$
With the same way for $1\leq j\leq \lfloor\frac{n}{2}\rfloor$, by applying the same argument we can show that the arranged sets
$B_j^{(1)}=\{x_j^{(1)}, x_{\lceil\frac{n}{2}\rceil+j}^{(1)}, x_c^{(1)}\}$ of vertices of $(C_n\Box P_k)\Box P_m$ cannot be doubly resolving sets for the $m$-cylinder graph $(C_n\Box P_k)\Box P_m$.
So, the minimum size of a doubly resolving set of vertices in $(C_n\Box P_k)\Box P_m$ is greater than $3$.
Now, let $D_i=A_i^{(1)}\cup x_c^{(m)}=\{x_i^{(1)}, x_{\lceil\frac{n}{2}\rceil+i-1}^{(1)}, x_c^{(1)}, x_c^{(m)}\}$ be an arranged subset of vertices of $(C_n\Box P_k)\Box P_m$, where $x_c^{(m)}$ lies in the layer $V_k^{(m)}$ of $(C_n\Box P_k)\Box P_m$.
We show that the  arranged subset $D_i$ is a doubly resolving set of vertices in $(C_n\Box P_k)\Box P_m$.
It is  enough to show that for every two vertices $x_t^{(r)}$ and $x_t^{(s)}$,  $1\leq t\leq nk$, $1\leq r<s\leq m$, $r(x_t^{(r)}|D_i)-r(x_t^{(s)}|D_i)\neq -\lambda I$, where $I$ indicates the unit $4$-vector $(1, ..., 1)$ and $\lambda$ is a positive integer.
For this purpose, let the distance between two the vertices $x_t^{(r)}$ and $x_t^{(s)}$ in $(C_n\Box P_k)\Box P_m$ is $\lambda$, then
we can verify that, $r(x_t^{(r)}|A_i^{(1)})-r(x_t^{(s)}|A_i^{(1)})= - \lambda I$, where $I$ indicates the unit $3$-vector, and $r(x_t^{(t)}|x_c^{(m)})-r(x_t^{(s)}|x_c^{(m)})= \lambda$.  Thus the minimum size of a doubly resolving set of vertices in $(C_n\Box P_k)\Box P_m$  is $4$.
\end{proof}
\begin{lemma}\label{dd.4}
If $n\geq 4$ is an even integer, then the minimum size of a doubly resolving set of vertices for the cylinder graph $C_n\Box P_k$ is $4$.
\end{lemma}
\begin{proof}
We first consider the cylinder graph $C_n\Box P_k$ with the vertex set $\{x_1,  ..., x_{nk}\}$ on the layers $V_1, V_2, ..., V_k$, is defined already.
Based on \cite{pap-jc-1}, if $n$ is an even integer, then the minimum size of a resolving set in $C_n\Box P_k$ is $3$. In particular,
using a similar argument in the proof of Theorem \ref{dd.1}, we can prove that any subset of vertices in $C_n\Box P_k$ of size $3$ cannot be a doubly resolving set for the cylinder graph $C_n\Box P_k$, and so the minimum size of a doubly resolving set of vertices for the cylinder graph $C_n\Box P_k$ is greater than $3$. Now, we claim that if $n$ is an even integer, then the minimum size of a doubly resolving set of vertices in $C_n\Box P_k$ is $4$. For this purpose, let $E_1$ be an arranged subset of vertices of $C_n\Box P_k$ so that $E_1$ is a minimal resolving set for the cylinder graph $C_n\Box P_k$ and all the elements of $E_1$ lie in exactly one of the layers $V_1$ or $V_k$. Without loss of generality if we consider  $E_1=\{x_1, x_{\frac{n}{2}}, x_{{\frac{n}{2}}+1} \}$, then we show that the arranged subset
$E_2=E_1\cup x_c=\{x_1, x_{\frac{n}{2}}, x_{{\frac{n}{2}}+1}, x_c\}$ of vertices in $C_n\Box P_k$,
where $x_c$ lies in the layer $V_k$ of $C_n\Box P_k$ and $x_c$ is a compatible vertex with respect to $x_1$, is one of the minimum  doubly resolving sets for the cylinder graph $C_n\Box P_k$. It is enough to show that for any compatible vertices $x_e$ and $x_d$ in $C_n\Box P_k$, $r(x_e|E_2)-r(x_d|E_2)\neq\lambda I$. Suppose $x_e\in V_p$ and $x_d\in V_q$ are compatible vertices in $C_n\Box P_k$, $1\leq p<q\leq k$. Hence, $r(x_e|E_1)-r(x_d|E_1)=-\lambda I$, where $\lambda$ is a positive integer, and $I$ indicates the unit $3$-vector $(1,1, 1)$. We also see that for the compatible vertex $x_c$ with respect to $x_1$, $r(x_e|x_c)-r(x_d|x_c)=\lambda$.
So, $r(x_e|E_2)-r(x_d|E_2)\neq\lambda I$, where $I$ indicates the unit $4$-vector $(1,..., 1)$, and the claim is proved.
\end{proof}
\begin{theorem}\label{dd.5}
If $n\geq4$ is an even integer, then the minimum size of a resolving set of vertices for the $m$-cylinder graph $(C_n\Box P_k)\Box P_m$ is $4$.
\end{theorem}
\begin{proof}
Suppose the $m$-cylinder graph $(C_n\Box P_k)\Box P_m$ is a graph with vertex set
$\{x_1^{(1)},  ..., x_{nk}^{(1)}\}\cup ...\cup \{x_1^{(m)},  ..., x_{nk}^{(m)}\}$
so that  the vertex $x_t^{(r)}$ is adjacent to $x_t^{(r+1)}$ in $(C_n\Box P_k)\Box P_m$, for  $1\leq t \leq nk$, and $1\leq r\leq m-1$.
Therefore, none of the minimal resolving sets of $C_n\Box P_k$ cannot be a resolving set for the $m$-cylinder graph $(C_n\Box P_k)\Box P_m$. So, the  minimum size of a resolving set of vertices in $(C_n\Box P_k)\Box P_m$ is greater than $3$. Now, we show that the  minimum size of a resolving set of vertices in $(C_n\Box P_k)\Box P_m$ is $4$. Let $x_1^{(1)}$ be a vertex in the layer $V_1^{(1)}$ of $(C_n\Box P_k)\Box P_m$ and $x_c^{(1)}$ be a compatible  vertex with respect to $x_1^{(1)}$, where $x_c^{(1)}$ lies in the layer $V_k^{(1)}$ of $(C_n\Box P_k)\Box P_m$. Based on Lemma \ref{dd.4}, we know that $1^{th}$ copy of the  arranged subset $E_2=\{x_1, x_{\frac{n}{2}}, x_{{\frac{n}{2}}+1}, x_c\}$, denoted by the set
$E_2^{(1)}=\{x_1^{(1)}, x_{\frac{n}{2}}^{(1)}, x_{{\frac{n}{2}}+1}^{(1)}, x_c^{(1)}\}$ is one of the minimum  doubly resolving sets for the cylinder graph $C_n\Box P_k$. Besides, the vertex $x_t^{(r)}$ is adjacent to the vertex $x_t^{(r+1)}$ in $(C_n\Box P_k)\Box P_m$,
and hence  the arranged set $E_2^{(1)}$  is one of  the resolving  sets for the $m$-cylinder graph $(C_n\Box P_k)\Box P_m$. Because for each vertex $x_t^{(i)}$ of $(C_n\Box P_k)\Box P_m$, we have
$$r(x_t^{(i)}|E_2^{(1)})=( d(x_t^{(1)}, x_1^{(1)})+i-1, d(x_t^{(1)}, x_{\frac{n}{2}}^{(1)})+i-1, d(x_t^{(1)}, x_{{\frac{n}{2}}+1}^{(1)})+i-1, d(x_t^{(1)}, x_{c}^{(1)})+i-1),$$
so all the vertices $\{x_1^{(1)},  ..., x_{nk}^{(1)}\}\cup ...\cup \{x_1^{(m)},  ..., x_{nk}^{(m)}\}$ of $(C_n\Box P_k)\Box P_m$  have various representations with respect to the set $E_2^{(1)}$. Thus the minimum size of a resolving set of vertices in $(C_n\Box P_k)\Box P_m$ is $4$.
\end{proof}
\begin{theorem}\label{dd.6}
If $n\geq4$ is an even integer, then the minimum size of a doubly resolving set of vertices for the $m$-cylinder graph
$(C_n\Box P_k)\Box P_m$ is $5$.
\end{theorem}
\begin{proof}
Suppose the $m$-cylinder graph $(C_n\Box P_k)\Box P_m$ is a graph with vertex set
$\{x_1^{(1)},  ..., x_{nk}^{(1)}\}\cup ...\cup \{x_1^{(m)},  ..., x_{nk}^{(m)}\}$
so that  the vertex $x_t^{(r)}$ is adjacent to $x_t^{(r+1)}$ in $(C_n\Box P_k)\Box P_m$, for  $1\leq t \leq nk$, and $1\leq r\leq m-1$.
By the proof of Theorem \ref{dd.5}, we know that the arranged set $E_2^{(1)}=\{x_1^{(1)}, x_{\frac{n}{2}}^{(1)}, x_{{\frac{n}{2}}+1}^{(1)}, x_c^{(1)}\}$ of vertices of $(C_n\Box P_k)\Box P_m$ is one of the resolving sets for the $m$-cylinder graph $(C_n\Box P_k)\Box P_m$, so that the arranged set $E_2^{(1)}$ cannot be a doubly resolving set for the $m$-cylinder graph $(C_n\Box P_k)\Box P_m$, and hence the minimum size of a doubly resolving set of vertices in
$(C_n\Box P_k)\Box P_m$ is greater than $4$.
Now, let $E_3=E_2^{(1)}\cup x_c^{(m)}=\{x_1^{(1)}, x_{\frac{n}{2}}^{(1)}, x_{{\frac{n}{2}}+1}^{(1)}, x_c^{(1)}, x_c^{(m)}\}$
be an arranged subset of vertices of $(C_n\Box P_k)\Box P_m$, where $x_c^{(m)}$  lies in the layer $V_k^{(m)}$ of $(C_n\Box P_k)\Box P_m$.
It is enough to show that for every two  vertices $x_t^{(r)}$ and $x_t^{(s)}$, $1\leq t\leq nk$, $1\leq r<s\leq m$, $r(x_t^{(r)}|E_3)-r(x_t^{(s)}|E_3)\neq -\lambda I$, where $I$ indicates the unit $5$-vector $(1, ..., 1)$ and $\lambda$ is a positive integer.
For this purpose, let the distance between two the vertices $x_t^{(r)}$ and $x_t^{(s)}$ in $(C_n\Box P_k)\Box P_m$ is $\lambda$. Then
we can verify that, $r(x_t^{(r)}|E_2^{(1)})-r(x_t^{(s)}|E_2^{(1)})= - \lambda I$, where $I$ indicates the unit $4$-vector, and $r(x_t^{(t)}|x_c^{(m)})-r(x_t^{(s)}|x_c^{(m)})= \lambda$. Therefore, the arranged subset $E_3$ is one of the minimum doubly resolving sets of vertices for the $m$-cylinder graph $(C_n\Box P_k)\Box P_m$. Thus the minimum size of a doubly resolving set of vertices in $(C_n\Box P_k)\Box P_m$  is $5$.
\end{proof}
\begin{theorem}\label{dd.7}
If $n\geq 3$ is an integer,  then the minimum size of a strong resolving set  of vertices for the $m$-cylinder graph $(C_n\Box P_k)\Box P_m$ is  $2n$.
\end{theorem}
\begin{proof}
Suppose the $m$-cylinder graph $(C_n\Box P_k)\Box P_m$ is a graph with vertex set
$\{x_1^{(1)},  ..., x_{nk}^{(1)}\}\cup ...\cup \{x_1^{(m)},  ..., x_{nk}^{(m)}\}$
so that the vertex $x_t^{(r)}$ is adjacent to $x_t^{(r+1)}$ in $(C_n\Box P_k)\Box P_m$, for  $1\leq t \leq nk$, and $1\leq r\leq m-1$. We know that each  vertex of  the layer  $V_1^{(1)}$ is maximally distant from a vertex of  the layer $V_k^{(m)}$ and each vertex of  the layer  $V_k^{(m)}$  is maximally distant from a vertex of the layer $V_1^{(1)}$. In particular, each  vertex of  the layer  $V_1^{(m)}$ is maximally distant from a vertex of the layer
$V_k^{(1)}$ and each  vertex of  the layer $V_k^{(1)}$  is maximally distant from a vertex of  the layer $V_1^{(m)}$, and so the minimum size of a strong resolving set of vertices for the $m$-cylinder graph $(C_n\Box P_k)\Box P_m$ is greater than or equal to $2n$. Because it is well known that for every pair of mutually maximally distant vertices $u$ and $v$ of a connected graph $G$ and for every strong metric basis $S$ of $G$, it follows that $u\in S$ or $v\in S$. Suppose the set $\{x_1^{(1)},  ...,  x_n^{(1)}\}$ is an arranged subset of vertices in  the layer $V_1^{(1)}$ of
$(C_n\Box P_k)\Box P_m$ and suppose that the set $\{x_1^{(m)},  ...,  x_n^{(m)}\}$ is an arranged subset of vertices in  the layer $V_1^{(m)}$ of $(C_n\Box P_k)\Box P_m$. Now, let $T=\{x_1^{(1)},  ...,  x_n^{(1)}\}\cup \{x_1^{(m)},  ...,  x_n^{(m)}\}$ be an arranged subset of vertices  of  $(C_n\Box P_k)\Box P_m$. In the following cases we show that the arranged set $T$, is one of the minimum strong resolving sets of vertices for the $m$-cylinder graph $(C_n\Box P_k)\Box P_m$. For this purpose let $x_e^{(i)}$ and $x_d^{(j)}$ be two various vertices of
$(C_n\Box P_k)\Box P_m$, $1\leq i,j\leq m$, $1\leq e,d\leq nk$ and $1\leq r\leq n$.\\

Case 1. If $i=j$, then $x_e^{(i)}$ and $x_d^{(i)}$ lie in $i^{th}$ copy of $(C_n\Box P_k)$ with vertex set $\{x_1^{(i)},  ..., x_{nk}^{(i)}\}$ so that $i^{th}$ copy of $(C_n\Box P_k)$ is a subgraph of $(C_n\Box P_k)\Box P_m$. Since $i=j$, then we can assume that $e<d$, because $x_e^{(i)}$ and
$x_d^{(i)}$ are various vertices.\\

Case 1.1. If both vertices $x_e^{(i)}$ and $x_d^{(i)}$ are compatible in $i^{th}$ copy of $(C_n\Box P_k)$ relative to $x_r^{(i)}\in V_1^{(i)}$, then there is the vertex $x_r^{(1)}\in V_1^{(1)}\subset T$ so that $x_e^{(i)}$ belongs to a shortest $x_r^{(1)}-x_d^{(i)}$ path, say as
$x_r^{(1)},..., x_r^{(i)},..., x_e^{(i)}, ..., x_d^{(i)}$.\\

Case 1.2. Suppose both vertices $x_e^{(i)}$ and $x_d^{(i)}$ are not compatible in $i^{th}$ copy of $(C_n\Box P_k)$, and lie
in various layers  or lie in the same layer in $i^{th}$ copy of $(C_n\Box P_k)$, also let
$x_r^{(i)}\in V_1^{(i)}$, be a compatible vertex relative to $x_e^{(i)}$. Hence there is the vertex $x_r^{(m)}\in V_1^{(m)}\subset T$ so that
$x_e^{(i)}$ belongs to a shortest $x_r^{(m)}-x_d^{(i)}$ path, say as
$x_r^{(m)},..., x_r^{(i)},..., x_e^{(i)}, ..., x_d^{(i)}$.\\

Case 2. If $i\neq j$, then $x_e^{(i)}$ lies in $i^{th}$ copy of $(C_n\Box P_k)$ with vertex set $\{x_1^{(i)},  ..., x_{nk}^{(i)}\}$, and $x_d^{(j)}$ lies in $j^{th}$ copy of $(C_n\Box P_k)$ with vertex set $\{x_1^{(j)},  ..., x_{nk}^{(j)}\}$. In this case we can assume that $i<j$.\\

Case 2.1. If $e=d$ and $x_r^{(j)}\in V_1^{(j)}$  is a compatible vertex relative to $x_d^{(j)}$, then there is the vertex $x_r^{(m)}\in V_1^{(m)}\subset T$ so that $x_d^{(j)}$ belongs to a shortest $x_r^{(m)}-x_e^{(i)}$ path, say as
$x_r^{(m)},..., x_r^{(j)},..., x_d^{(j)}, ..., x_e^{(i)}$.\\

Case 2.2. If $e< d$, also $x_e^{(i)}$ and $x_d^{(j)}$ lie in various layers of $(C_n\Box P_k)\Box P_m$ or $x_e^{(i)}$ and $x_d^{(j)}$ lie in the same layer of $(C_n\Box P_k)\Box P_m$ and $x_r^{(i)}\in V_1^{(i)}$ is a compatible vertex relative to $x_e^{(i)}$, then there is the vertex
$x_r^{(1)}\in V_1^{(1)}\subset T$ so that  $x_e^{(i)}$ belongs to a shortest $x_r^{(1)}-x_d^{(j)}$ path, say as
$x_r^{(1)},..., x_r^{(i)},..., x_e^{(i)}, ..., x_d^{(j)}$.\\

Case 2.3. If $e> d$, also $x_e^{(i)}$ and $x_d^{(j)}$ lie in various layers of $(C_n\Box P_k)\Box P_m$ or $x_e^{(i)}$ and $x_d^{(j)}$ lie in the same layer of $(C_n\Box P_k)\Box P_m$ and $x_r^{(j)}\in V_1^{(j)}$ is a compatible vertex relative to  $x_d^{(j)}$, then there
is the vertex $x_r^{(m)}\in V_1^{(m)}\subset T$ so that  $x_d^{(j)}$ belongs to a shortest $x_r^{(m)}-x_e^{(i)}$ path, say as
$x_r^{(m)},..., x_r^{(j)},..., x_d^{(j)}, ..., x_e^{(i)}$.
\end{proof}


\section{Results for the graph $L(n)$ }

Let  $n$  be a fixed positive integer, with $n\geq5$  and $[n]=\{1, ..., n\}$  and let $i,j \in[n], i\neq j, i<j,1\leq  i\leq n-1, 2\leq  j\leq n$. 
Suppose that $G$ is a graph with vertex set $W_1\cup ...\cup W_n$, where
$W_r=\{\{v_r, v_iv_j\}\,\ | \,\,  v_r=v_i \,\ \text{or} \,\ v_r=v_j \}$, for $1\leq r\leq n$. We say that two various vertices
$\{v_r, v_iv_j\}$ and $\{v_k, v_pv_q\}$ are adjacent in $G$ if and only if $v_r=v_k$ or $v_iv_j=v_pv_q$, that is $i=p$ and $j=q$. It is not hard to see that this family  of  graphs is isomorphic to the the graph $L(n)$, is defined in 
\cite{pap-sm-1}. Based on
\cite{pap-sm-1}
 we can see that $L(n)$ is a connected vertex transitive graph of valency $n-1$, with diameter $3$, and the order $n(n-1)$. It is easy to see that every $W_r$ is a  clique of size $n-1$ in the graph $L(n)$. We also undertake the necessary task of introducing some of the basic notation for this class of graphs. We say that two  cliques $W_r$ and $W_k$ are adjacent in $L(n)$, if there is a vertex in  clique $W_r$  so that this vertex is adjacent to exactly one vertex of  clique $W_k$, $r,k \in [n], r\neq k$. Also, for any  clique $W_r$  in $G=L(n)$ we use $N(W_r)=\bigcup_{w\in W_r}N_G(w)$ to denote the vertices in the all cliques $W_k$, say $w_k$,
$1\leq k\leq n$ and $k\neq r$ so that $w_k$ is adjacent to one vertex of  the clique $W_r$. In this section, we first  obtain an upper bound in the next Lemma for the resolving set of the graph $L(n)$, and we show that the minimum size of a resolving set in graph $L(n)$ is $n-2$, see Theorem \ref{l.2}. Moreover, we study some resolving sets with the minimum size for the graph $L(n)$. 
\begin{lemma}\label{l.1}
Let $n$  be a fixed positive integer, with $n\geq5$. If $1\leq r\leq n$, then each subset of $N(W_r)$ of size at least $n-2$ can be a resolving set for $L(n)$.
\end{lemma}
\begin{proof}
Suppose that $V(L(n))=W_1\cup ...\cup W_n$, where the set
$W_r=\{\{v_r, v_iv_j\}\,\ | \,\,  v_r=v_i \,\ \text{or} \,\ v_r=v_j \}$ denote a  clique of size $n-1$ in the graph $L(n)$. We know that $N(W_r)$  indicate the vertices in the all  cliques $W_k$, say $w_k$, $1\leq k\leq n$ and $k\neq r$ so that $w_k$ is adjacent  to one vertex of  the  clique $W_r$, also we can see that the cardinality of $N(W_r)$ is $n-1$.
Since $L(n)$ is a vertex transitive graph, then without loss of generality we may consider the   clique $W_1$. So
$N(W_1)= \{y_2, ... , y_{n-1}, y_n\}$, where $y_k= \{v_k, v_1v_k\}\in W_k$. If we consider $C_1=N(W_1)-\{y_{n-1}, y_n\}= \{y_2, ... , y_{n-2}\}$, then
there are exactly two vertices $\{v_1, v_1v_{n-1}\}, \{v_1, v_1v_n\}\in W_1$ so that $r(\{v_1, v_1v_{n-1}\}|C_1)=r(\{v_1, v_1v_n\}|C_1)= (2,..., 2)$, and so we can verify that the set $C_1$ cannot be a resolving set for $L(n)$, and so
any subset of $N(W_r)$ of size  $n-3$ cannot be a resolving set for $L(n)$.
Now, we take $C_2=N(W_1)-y_n=\{y_2, ... , y_{n-1}\}$ and we show that all the vertices in $V(L(n))\setminus C_2$ have different representations relative to $C_2$. If we consider the vertex $\{v_1, v_1v_{n-1}\}\in W_1$, then the vertex $\{v_1, v_1v_{n-1}\}$ is adjacent to the vertex $y_{n-1}\in W_{n-1}$, and hence $r(\{v_1, v_1v_{n-1}\}|C_2)\neq r(\{v_1, v_1v_n\}|C_2)$. Also,  every vertex $w$ in the clique  $ W_1$ is adjacent to exactly a vertex of  each  clique  $W_j$, $2\leq j\leq n$. So, all the vertices $ w\in W_1$  have various metric representations  relative to the subset $C_2$. In particular, for every vertex  $w\in W_r$, $2\leq r\leq n-1$ so that $w\notin N(W_1)$ and each $y_s\in C_2 $, $2\leq s\leq n-1$, if $w, y_s$ lie in a  clique  $ W_s$, $2\leq s\leq n-1$, then we have $d(w, y_s)=1$; otherwise  $d(w, y_s)\geq 2$.
Moreover, all the vertices in the clique $W_n$ have various metric representations relative to the subset $C_2$, because for every vertex $w$ in the  clique $W_n$  such that  $w$ is not equal to the vertex $\{v_n, v_1v_n\}$ in the  clique $W_n$, there is  exactly one element
$y_s\in C_2$ such that $d(w, y_s)=2$; otherwise $d(w, y_s)>2$, $2\leq s\leq n-1$. Finally, for the vertex $y_n=\{v_n, v_1v_n\}$ in  the  clique $W_n$ and every element $y_s\in C_2$ we have $d(w, y_s)=3$. Thus the arranged  subset $C_2= \{y_2, ... , y_{n-1}\}$ of vertices in  $L(n)$ is a resolving set for  $L(n)$ of size  $n-2$, and so each subset of $N(W_r)$  of size  $n-2$ is a resolving set for $L(n)$. 
\end{proof}
\begin{theorem}\label{l.2}
If $n\geq 5$ is a fixed positive integer, then the minimum size of a resolving set in graph $L(n)$ is $n-2$.
\end{theorem}
\begin{proof}
Suppose that $V(L(n))=W_1\cup ...\cup W_n$, where the set
$W_r=\{\{v_r, v_iv_j\}\,\ | \,\,  v_r=v_i \,\ \text{or} \,\ v_r=v_j \}$  indicate  a  clique of size $n-1$ in the graph $L(n)$ for $1\leq r\leq n$. Let $D_1=\{W_1, W_2, ..., W_{n-3}\}$ be a subset of vertices of $L(n)$ of size $(n-1)(n-3)$, consisting of some of the  cliques of $L(n)$ and let $D_2=\{W_{n-2}, W_{n-1}, W_{n}\}$ be a subset of vertices  of $L(n)$, consisting of exactly three  cliques of $L(n)$. Now, let $D_3$ be a subset of $D_2$, consisting of exactly one clique of $D_2$, say $W_{n}$, and let $D_3=\{W_n\}$. Thus there are exactly two distinct vertices in $D_3=\{W_n\}$ say $x$ and  $y$ so that $x$ is adjacent to a vertex of $W_{n-1}$ and $y$ is adjacent to a vertex of $W_{n-2}$, and hence the metric representations of two vertices $x$  and $y$ are identical relative to $D_1$. Therefore, the set  $D_1=\{W_1, W_2, ..., W_{n-3}\}$ cannot be a resolving set for the graph $L(n)$, and so, if $w_k$ is an arbitrary vertex of $L(n)$ such that $w_k\in W_k$  for $1\leq  k\leq n-3$, then the subset  $\{w_1, w_2, ..., w_{n-3}\}$ of vertices of  $L(n)$
cannot be a resolving set for graph $L(n)$. So the cardinality of a minimum resolving set for graph $L(n)$ must be greater
than or equal to $n-2$. In particular, based on the Lemma \ref{l.1}, the arranged  subset $C_2=N(W_1)-y_n=\{y_2, ... , y_{n-1}\}$ of vertices in  $L(n)$ is a resolving set for  $L(n)$ of size  $n-2$, and so the  minimum size of a resolving set in  graph $L(n)$ is  $n-2$.
\end{proof}
\begin{lemma} \label{l.3}
Consider the graph $L(n)$ with vertex set $W_1\cup ...\cup W_n$ for $n\geq 5$. Any subset of $N(W_r)$ of size  $n-2$ cannot be a doubly resolving set for $L(n)$.
\end{lemma}
\begin{proof}
Since $L(n)$ is a vertex transitive graph, then without loss of generality we may consider the  clique $W_1$.
So if we take $C_2=N(W_1)-y_n=\{y_2, ..., y_{n-1}\}$, where for $2\leq k\leq n$ we have  $y_k= \{v_k, v_1v_k\}\in V_k$, then
by Lemma \ref{l.1} and Theorem \ref{l.2}, the subset
$C_2=N(W_1)-y_n=\{y_2, ..., y_{n-1}\}$ of vertices in  $L(n)$ is a minimum resolving set for  $L(n)$ of size $n-2$.
Hence by considering the vertices $u=\{v_1, v_1v_n\}\in W_1$ and $y_n=\{v_n, v_1v_n\}\in W_n$, we see that
$d(u, r) - d(u, s) = d(y_n, r) - d(y_n, s)$ for elements $r, s \in C_2$,
because for each element $z\in C_2$ we have $d(u, z)=2$ and $d(y_n, z)=3$. Thus the subset
$C_2=N(W_1)-y_n=\{y_2, ..., y_{n-1}\}$ of vertices in  $L(n)$ cannot be a doubly resolving set for $L(n)$, and so
any subset  $N(W_r)$ of graph $L(n)$ of size  $n-2$ cannot be a doubly resolving set for $L(n)$.
\end{proof}
\begin{theorem}\label{l.4}
If $n\geq 5$ is a fixed positive integer, then the minimum size of a doubly resolving set in graph $L(n)$ is $n-1$.
\end{theorem}
\begin{proof}
Suppose that $V(L(n))=W_1\cup ...\cup W_n$,
where $W_r=\{\{v_r, v_iv_j\}\,\ | \,\,  v_r=v_i \,\ \text{or} \,\ v_r=v_j \}$.
By Lemma \ref{l.1}, and Theorem \ref{l.2},  the subset
$C_2=N(W_1)-y_n=\{y_2, ..., y_{n-1}\}$ of vertices in  $L(n)$ is a minimum resolving set for  $L(n)$ of size $n-2$, where $y_k= \{v_k, v_1v_k\}\in W_k$ for $2\leq k\leq n$. Also, from Lemma \ref{l.3} we know that the subset $C_2$ is not a doubly resolving set for $L(n)$, and hence the  minimum size of a doubly resolving set in  $L(n)$ is  greater than or equal to $n-1$.
Now, if we take $C_3=N(W_1)= \{y_2, ... , y_{n-1}, y_n\}$,  where $y_k= \{v_k, v_1v_k\}\in W_k$, then based on Lemma \ref{l.1},
we know that the subset
$C_3=N(W_1)=\{y_2, ...,  y_{n-1}, y_n\}$ of vertices in $L(n)$ is a resolving set for $L(n)$ of size $n-1$. We show that $C_3$ is a doubly resolving set for $L(n)$. It will be enough to show that for any two various vertices $u$ and $v$ in $L(n)$, there exist elements $x$ and  $y$ from $C_3$ so that
$d(u, x) - d(u, y) \neq d(v, x) - d(v, y)$. Consider two vertices $u$ and $v$ in  $L(n)$.
Then the result can be deduced from  the following  cases: \newline

Case 1.
Suppose, both vertices $u$ and $v$ lie in the  clique $W_1$. Hence, there exists an element $x\in C_3$ so that $x\in W_r$ and  $x$ is adjacent to $u$, also, there exists an element $y\in C_3$ so that $y\in W_k$ and  $y$ is adjacent to $v$ for some $r,k\in [n]-1$, $r\neq k$; and hence
$-1=1-2=d(u, x) - d(u, y) \neq d(v, x) - d(v, y)=2-1=1$.\newline

Case 2.
Suppose, both vertices $u$ and $v$ lie in the  clique $W_r$, $r\in[n]-1$, so that $u, v\notin C_3$.  Hence, there exists an element $x\in C_3$ so that $x\in W_r$ and $d(u,x)=d(v, x)=1$, also there exists an element $y\in C_3$ so that $y\in W_k$, $r\neq k$, and  $d(u, y)=2$, $d(v, y)=3$ or $d(u, y)=3$, $d(v, y)=2$. Thus  $d(u, x) - d(u, y) \neq d(v, x) - d(v, y).$\newline

Case 3.
Suppose that $u$ and $v$  are two distinct vertices in $L(n)$ so that $u\in W_1$ and $v\in W_r$, $r\in[n]-1$. Hence $d(u,v)=t$, for $1\leq t\leq 3$. \newline

Case 3.1. If $t=1$, then  $v\in C_3$. So if we consider $x=v$ and $v\neq y\in C_3$, then we have  $d(u, x) - d(u, y) \neq d(v, x) - d(v, y)$.\newline

Case 3.2. If $t=2$, then  in this case may be $v\in C_3$ or $v\notin C_3$. If $v\in C_3$, then there  exists an element $x\in C_3$ so that $x\in W_k$,
$k\in [n]-1$, $r\neq k$ and $d(u, x)=1,  d(v, x)=3$. So if we consider $v=y$, then we have
$-1=1-2=d(u, x) - d(u, y) \neq d(v, x) - d(v, y)=3-0=3$.
If $v\notin C_3$,  then there exists an element $x\in W_r$ so that $x\in C_3$ and $d(u, x)=d(v, x)=1$, also there exists an element $y\in C_3$ so that
$y\in W_k$, $k\in[n]-\{1, r\}$, and $d(u, y)=2$, $d(v, y)=3$ or $d(u, y)=3$, $d(v, y)=2$, and hence we have
$d(u, x) - d(u, y) \neq d(v, x) - d(v, y)$.\newline

Case 3.3. If $t=3$, then there exists an element $x\in W_r$ so that $x\in C_3$ and $d(u, x)=2$, $d(v, x)=1$, also there exists an element $y\in C_3$ so that $y\in W_k$, $k\in[n]-\{1, r\}$, and  $d(u, y)=1$, $d(v, y)=3$, and hence we have
$d(u, x) - d(u, y) \neq d(v, x) - d(v, y)$.
\newline

Case 4.
Suppose that $u$ and $v$ are two distinct vertices in $L(n)$ so that $u\in W_r$ and $v\in W_k$, $r, k\in[n]-1$, $r\neq k$. If both two vertices
$u$ and  $v$ lie in $C_3$ or exactly one of them vertices lies in $C_3$, then there is nothing to prove. Now suppose that both two vertices
$u, v \notin C_3$. Hence there exist  elements  $x\in C_3$ and  $y\in C_3$ so that $x\in W_r$ and $y\in W_k$, and hence we have
$d(u, x) - d(u, y) \neq d(v, x) - d(v, y)$.\newline
\end{proof}
\begin{proposition}\label{l.5}
If $n\geq 5$ is a fixed positive integer, then for $1\leq r\leq n$, any set  $N(W_r)$ of size $n-1$ cannot be a strong resolving set for $L(n)$.
\end{proposition}
\begin{proof}
Suppose that $V(L(n))=W_1\cup...\cup W_n$, where the set
$W_r=\{\{v_r, v_iv_j\}\,\ | \,\,  v_r=v_i \,\ \text{or} \,\ v_r=v_j \}$ indicate a  clique of size $n-1$ in the graph $L(n)$ for $1\leq r\leq n$. Without loss of generality if we consider the  clique $W_1$ and we take
$C_3=N(W_1)= \{y_2, ... , y_{n-1}, y_n\}$, where for $2\leq k\leq n$ we have  $y_k= \{v_k, v_1v_k\}\in W_k$, then
By Lemma \ref{l.1},  we know that for the  clique $W_1$ in  $L(n)$, the subset
$C_3=N(W_1)=\{y_2, ...,  y_{n-1}, y_n\}$ of vertices in  $L(n)$ is a  resolving set for $L(n)$ of size $n-1$.
By considering various vertices $ w_1\in W_r$ and $w_2\in W_k$, $1 <r, k <n$, $r\neq k$, so that   $d(w_1, w_2)=3$ and $w_1, w_2\notin C_3$,  there is not a $y_r\in C_3$ so that $w_1$ belongs to a shortest $w_2 - y_r$ path or $w_2$ belongs to a shortest $w_1 - y_r$ path.
Thus the set  $C_3=N(W_1)=\{y_2, ...,  y_{n-1}, y_n\}$ cannot be a strong resolving set for  $L(n)$, and so
any  set $N(W_r)$ of graph $L(n)$  of size  $n-1$ cannot be a strong  resolving set for $L(n)$.\newline
\end{proof}
\begin{theorem}\label{l.6}
If $n\geq $ is a fixed positive integer, then the minimum size of a strong resolving set in  graph $L(n)$ is $n(n-2)$.
\end{theorem}
\begin{proof}
Suppose that $V(L(n))=W_1\cup...\cup W_n$, where the set
$W_r=\{\{v_r, v_iv_j\}\,\ | \,\,  v_r=v_i \,\ \text{or} \,\ v_r=v_j \}$ indicate a  clique of size $n-1$ in graph $L(n)$ for $1\leq r\leq n$.
Without loss of generality if we consider the vertex $\{v_1, v_1v_2\}$ in the  clique $W_1$, then there are exactly $(n-2)$ vertices in any  cliques $W_3, W_4, ..., W_n$, so that the distance between the vertex $\{v_1, v_1v_2\}\in W_1$ and these vertices in any cliques
$W_3, W_4, ..., W_n$ is $3$, and hence these vertices must lie in every minimal strong resolving set of $L(n)$. Note that the cardinality of these vertices is $(n-2)(n-2)$. On the other hand if we take $C_3=N(W_1)= \{y_2, ... , y_{n-1}, y_n\}$, where for $2\leq k\leq n$ we have
$y_k= \{v_k, v_1v_k\}\in W_k$, then  the distance between two distinct vertices of  $N(W_1)$  is $3$, and so  $n-2$ vertices of $N(W_1)$ must lie in every minimal strong resolving set of $L(n)$, we may consider these vertices are $y_3, ... , y_{n-1}, y_n$. If we now consider  the  cliques $W_1$ and $W_2$, then there are exactly $(n-2)$ vertices in the clique $W_1$, so that the distance these vertices  from $(n-2)$ vertices in the  clique $W_2$ is $3$, and hence we may assume that $(n-2)$ vertices  of the  clique $W_2$ so that the distance between  these vertices from $(n-2)$ vertices  of $W_1$  is $3$, must lie in every minimal strong resolving set of $L(n)$. Thus the minimum size of a strong resolving set in the graph $L(n)$ is $n(n-2)$.
\end{proof}

\bigskip
{\footnotesize

\noindent \textbf{Acknowledgements}\\
We would like to express our gratitude to Prof. Jia-Bao Liu, who gave us help and insightful suggestions  provided a great improvement to our paper.  \\[2mm]
\noindent \textbf{Authors' informations}\\
Ali Zafari${}^{a}$(\textsc{Corresponding Author})
(\url{zafari.math@pnu.ac.ir;zafari.math.pu@gmail.com})\\
 Saeid Alikhani${}^b$
(\url{alikhani@yazd.ac.ir})\\
\noindent ${}^{a}$ Department of Mathematics, Faculty of Science,
Payame Noor University, P.O. Box 19395-4697, Tehran, Iran.\\
${}^{b}$ Department of Mathematical Sciences, Yazd University, 89195-741, Yazd, Iran.

{\footnotesize


\begin{thebibliography}{5}
\bibitem{pap-M.Abas.ET.AL}{M. Abas and  T. Vetr\'{i}k},
\emph{Metric dimension of Cayley digraphs of split metacyclic groups},
{Theoretical Computer Science},  809 (2020)  61-72.
\bibitem{pap-aa.s.s1-2}{A. Ahmad and  S. Sultan},
\emph{On Minimal Doubly Resolving Sets of Circulant Graphs},
{Acta Mechanica Slovaca},  21(1) (2017) 6–11.
\bibitem{pap-M.Ahmad.ET.AL} {M. Ahmad,  D. Alrowaili,  Z. Zahid, I. Siddique and  A. Iampan},
\emph{Minimal Doubly Resolving Sets of Some Classes of Convex Polytopes},
{Journal of Mathematics},  2022 (2022) 1-13.
\bibitem{pap-M.Ali.ET.AL} {M. Ali, G. Ali, M. Imran, A. Q. Baig and  M. K. Shafiq},
\emph{On the metric dimension of Mobius ladders},
{Ars Combinatoria}., 105 (2012)  403-410.
\bibitem{pap-M.Baca.ET.AL} {M. Baca, E. T. Baskoro,  A. N. M. Salman, S. W. Saputro and  D. Suprijanto},
\emph{The metric dimension of regular bipartite graphs},
{Bulletin mathematique de la Societe des Sciences Mathematiques de Roumanie Nouvelle Serie},  54 (2011) 15-28.
\bibitem{pap-psb.gh1-2}{P. S. Buczkowski, G. Chartrand,  C. Poisson and  P. Zhang},
\emph{On k-dimensional graphs and their bases},
{Periodica Mathematica Hungarica},  46(1) (2003)  9–15.
\bibitem{pap-jc-1} {J. C\'{a}ceres, C. Hernando, M. Mora, I. M. Pelayo,  M. L. Puertas,  C. Seara and D. R. Wood},
\emph{On the metric dimension of Cartesian products of graphs},
{SIAM Journal of Discrete Mathematics}.,  21(2) (2007)  423–441.
\bibitem{pap-pj.jh-1} {P. J. Cameron and J. H. Van Lint},
\emph{Designs, Graphs, Codes and Their Links},
{London Mathematical Society Student Texts 22, Cambridge: Cambridge University Press}, 1991.
\bibitem{pap-chr-1} {G. Chartrand, L. Eroh, M. A. Johnson  and  O. R. Ollermann},
\emph{Resolvability in graphs and the metric dimension of a graph},
{Discrete Applied  Mathematics},  105 (2000)  99–113.
\bibitem{pap-X. Chen1-2}{X. Chen,  X. Hu and C. Wang},
\emph{Approximation for the minimum cost doubly resolving set problem},
{Theoretical Computer Science},  609(3) (2016) 526-543.
\bibitem{pap-M.Feng}{M. Feng,   M. Xu and   K. Wang},
\emph{On the metric dimension of line graphs},
{Discrete Applied Mathematics}, 161 (2013) 802-805.
\bibitem{pap-cg-1} {C. Godsil and  G. Royle},
\emph{Algebraic Graph Theory}, Springer, New York, NY, USA, 2001.
\bibitem{pap-fh-1} {F. Harary and  R. A. Melter},
\emph{On the metric dimension of a graph},
{Ars Combinatoria},   2 (1976) 191–195.
\bibitem{pap-mi-1} {M. Imran,  A. Q. Baig, and  A. Ahmed},
\emph{Families of plane graphs with constant metric dimension},
{Utilitas Mathematica},  88 (2012)   43-57.
\bibitem{pap-mj-1} {M. Jannesari and B. Omoomi},
\emph{The metric dimension of the lexicographic product of graphs},
{Discrete Mathematics},  312(22) (2012) 3349-3356.
\bibitem{pap-mj-1.1} {M. Jannesari}, 
\emph{On doubly resolving sets in graphs}, 
Bulletin of Malaysian Mathematical Society, 45 (2022) 2041–2052.
\bibitem{pap-Khuller} {S. Khuller, B. Raghavachari and  A. Rosenfeld},
\emph{Landmarks in graphs},
{Discrete Applied Mathematics}, vol. 70(3) (1996) 217–229.
\bibitem{pap-skbp} {S. Klavzar, B. Patkos, G. Rus and I. G. Yero},
\emph{On general position sets in Cartesian products},
{Results in Mathematics }, 76:123 (2021)  1-21.
\bibitem{pap-J.Kratica.ET.AL1} {J. Kratica,  M. Cangalovic and  V. Kovacevic-Vujcic},
\emph{Computing minimal doubly resolving sets of graphs},
{Computers \& Operations Research},  36 (2009) 2149–2159.
\bibitem{pap-J.Kratica.ET.AL} {J. Kratica, V. Kovacevic-Vujcic,   M. Cangalovic and   M. Stojanovic},
\emph{Minimal doubly resolving sets and the strong metric dimension of Hamming graphs},
{Applicable Analysis and Discrete Mathematics}, vol. 6(1)(2012)  63–71.
\bibitem{pap-Kuziak} {D. Kuziak, I. G. Yero, and J. A. Rodr\'{i}guez-Vel\'{a}zquez},
\emph{On the strong metric dimension of corona product graphs and join graphs},
{Discrete Applied Mathematics},  161 (2013)   1022-1027.
\bibitem{pap-J.B.Liu.ET.AL}{J. B. Liu, M. F .Nadeem, H. M. A. Siddiqui and  W. Nazir},
\emph{Computing Metric Dimension of Certain Families of Toeplitz Graphs},
{IEEE Access},  7 (2019)  126734-126741.
\bibitem{pap-jb1-1}{J. B. Liu and  A. Zafari},
\emph{Computing Minimal Doubly Resolving Sets and the Strong Metric Dimension of the Layer Sun Graph and the Line Graph of the Layer Sun Graph },
{Complexity}, (2020) 1-8.
\bibitem{pap-sm-1}{ S. M. Mirafzal},
\emph{A new class of integral graphs constructed from the hypercube},
{Linear Algebra and its Applications},  558 (2018) 186-194.
\bibitem{pap-kn-1} {K. Nie and K. Xu},
\emph{The doubly metric dimension of cylinder graphs and torus graphs},
\emph{Bulletin of the Malaysian Mathematical Sciences Society}, 46, 2023.
\bibitem{pap-oro-1} {O. R. Oellermann and J. Peters-Fransen},
\emph{The strong metric dimension of graphs and digraphs},
{Discrete Applied Mathematics},  155 (2007) 356-364.
\bibitem{pap-ja-1} {J. A. Rodr\'{i}guez-Vel\'{a}zquez, I. G. Yero,   D. Kuziak and   O. R. Oellermann},
\emph{On the strong metric dimension of Cartesian and direct products of graphs},
{Discrete Mathematics}, 335 (2014) 8-19.
\bibitem{pap-pjs-1} {P. J. Slater},
\emph{Leaves of trees},
{Congressus Numerantium}, 14 (1975) 549–559.
\bibitem{pap-as-1} {A. Sebo and  E. Tannier},
\emph{On metric generators of graphs},
{Mathematics of Operations Research}, vol. 29(2) (2004) 383–393.
\bibitem{pap-az-1} {A. Zafari, N. Habibi and S. Alikhani},
\emph{Resolving sets of vertices with the minimum size in graphs},
{Journal of Mathematics and Society}, 8(3) (2023) 41-54.
\bibitem{pap-X.Zhang.ET.AL} {X. Zhang and M. Naeem},
\emph{Metric Dimension of Crystal Cubic Carbon Structure},
{Journal of Mathematics}, (2021) 1-8.
\end{thebibliography}
\end{document}